\def\stoc{0} 

\documentclass[11pt]{article}

\usepackage{hyperref}
\usepackage{url}
\usepackage{fullpage}
\usepackage{amsfonts, amsmath, amssymb}

\begin{document}

\hbadness=10000
\vbadness=10000

\setlength{\parskip}{\medskipamount}

\newtheorem{theorem}{Theorem}
\newtheorem{corollary}[theorem]{Corollary}
\newtheorem{lemma}[theorem]{Lemma}
\newtheorem{observation}[theorem]{Observation}
\newtheorem{proposition}[theorem]{Proposition}
\newtheorem{definition}[theorem]{Definition}
\newtheorem{claim}[theorem]{Claim}
\newtheorem{fact}[theorem]{Fact}
\newtheorem{assumption}[theorem]{Assumption}

\newcommand{\qed}{\rule{7pt}{7pt}}

\newenvironment{proof}{\noindent{\bf Proof}\hspace*{1em}}{\qed\bigskip}
\newenvironment{proof-sketch}{\noindent{\bf Sketch of Proof}\hspace*
{1em}}{\qed\bigskip}
\newenvironment{proof-idea}{\noindent{\bf Proof Idea}\hspace*{1em}}
{\qed\bigskip}
\newenvironment{proof-of-lemma}[1]{\noindent{\bf Proof of Lemma #1}
\hspace*{1em}}{\qed\bigskip}
\newenvironment{proof-attempt}{\noindent{\bf Proof Attempt}\hspace*
{1em}}{\qed\bigskip}
\newenvironment{proofof}[1]{\noindent{\bf Proof
of #1:}}{\qed\bigskip}
\newenvironment{remark}{\noindent{\bf Remark}\hspace*{1em}}{\bigskip}

\makeatletter
\def\fnum@figure{{\bf Figure \thefigure}}
\def\fnum@table{{\bf Table \thetable}}
\long\def\@mycaption#1[#2]#3{\addcontentsline{\csname
   ext@#1\endcsname}{#1}{\protect\numberline{\csname
   the#1\endcsname}{\ignorespaces #2}}\par
   \begingroup
     \@parboxrestore
     \small
     \@makecaption{\csname fnum@#1\endcsname}{\ignorespaces #3}\par
   \endgroup}
\def\mycaption{\refstepcounter\@captype \@dblarg{\@mycaption\@captype}}
\makeatother

\newcommand{\figcaption}[1]{\mycaption[]{#1}}
\newcommand{\tabcaption}[1]{\mycaption[]{#1}}
\newcommand{\head}[1]{\chapter[Lecture \##1]{}}
\newcommand{\mathify}[1]{\ifmmode{#1}\else\mbox{$#1$}\fi}
\newcommand{\bigO}O
\newcommand{\set}[1]{\mathify{\left\{ #1 \right\}}}
\def\half{\frac{1}{2}}


\newcommand{\enc}{{\sf Enc}}
\newcommand{\dec}{{\sf Dec}}
\newcommand{\E}{{\rm Exp}}
\newcommand{\Var}{{\rm Var}}
\newcommand{\Z}{{\mathbb Z}}
\newcommand{\F}{{\mathbb F}}
\newcommand{\K}{{\mathbb K}}
\newcommand{\N}{{\mathbb N}}
\newcommand{\integers}{{\mathbb Z}^{\geq 0}}
\newcommand{\R}{{\mathbb R}}
\newcommand{\Q}{{\cal Q}}
\newcommand{\eqdef}{{\stackrel{\rm def}{=}}}
\newcommand{\from}{{\leftarrow}}
\newcommand{\vol}{{\rm Vol}}
\newcommand{\poly}{{\rm {poly}}}
\newcommand{\ip}[1]{{\langle #1 \rangle}}
\newcommand{\wt}{{\rm {wt}}}
\renewcommand{\vec}[1]{{\mathbf #1}}
\newcommand{\mspan}{{\rm span}}
\newcommand{\rs}{{\rm RS}}
\newcommand{\RM}{{\rm RM}}
\newcommand{\Had}{{\rm Had}}
\newcommand{\calc}{{\cal C}}

\newcommand{\fig}[4]{
         \begin{figure}
         \setlength{\epsfysize}{#2}
         \vspace{3mm}
         \centerline{\epsfbox{#4}}
         \caption{#3} \label{#1}
         \end{figure}
         }

\newcommand{\ord}{{\rm ord}}
\def\blfootnote{\xdef\@thefnmark{}\@footnotetext}
\providecommand{\norm}[1]{\lVert #1 \rVert}
\newcommand{\embed}{{\rm Embed}}
\newcommand{\qembed}{\mbox{$q$-Embed}}
\newcommand{\calh}{{\cal H}}
\newcommand{\lp}{{\rm LP}}
\setlength{\parindent}{0in}

\newcommand{\Rej}{{\mathsf{Rej}}}

\newcommand{\rmdn}{{\RM(d,n)}}
\newcommand{\calf}{{\cal F}}
\newcommand{\calp}{{\cal P}}
\newcommand{\bad}{\mbox{\sc Bad}}
\newcommand{\polylog}{{\rm polylog}}

\newcommand{\takklr}{T_{\rm AKKLR}}
\newcommand{\tgowers}{T_{\rm GN}}
\newcommand{\tgowersk}{T_{{\rm GN}(k)}}
\newcommand{\tstar}{\tgowers}
\newcommand{\deltamin}{\delta_{\min}}
\newcommand{\ignore}[1]{}

\title{Optimal Testing of Reed-Muller Codes}
\author{%
Arnab Bhattacharyya\thanks{Computer Science and Artificial
  Intelligence Laboratory, MIT, {\tt abhatt@mit.edu}.  Work partially
  supported by a DOE Computational Science Graduate Fellowship and NSF
Awards 0514771, 0728645, and 0732334.}
\and Swastik Kopparty\thanks{Computer Science and Artificial
  Intelligence Laboratory, MIT, {\tt swastik@mit.edu}. Work was partially done while author was a summer intern at Microsoft Research New England and partially supported by NSF Grant CCF-0829672.}
\and Grant Schoenebeck\thanks{Department of Computer Science, University of California-Berkeley, {\tt grant@cs.berkeley.edu}.  Work was partially done while author was a summer intern at Microsoft Research New England and partially supported by a National Science Foundation Graduate Fellowship.}
\and Madhu Sudan\thanks{Microsoft Research, One Memorial Drive, Cambridge, MA 02142, USA, {\tt madhu@mit.edu}.}
\and David Zuckerman\thanks{Computer Science Department,
University of Texas at Austin, {\tt diz@cs.utexas.edu}.
Work was partially done while the author consulted at Microsoft Research
New England, and partially supported by NSF Grants CCF-0634811
and CCF-0916160.}}

\maketitle

\begin{abstract}
We consider the problem of testing if a given function
$f : \F_2^n \rightarrow \F_2$ is close to any degree $d$ polynomial
in $n$ variables, also known as the Reed-Muller testing problem.
The Gowers norm is based on a natural $2^{d+1}$-query test for this property.
Alon et al.~\cite{AKKLR} rediscovered this test and showed that it accepts
every degree $d$ polynomial with probability $1$, while it rejects
functions that are $\Omega(1)$-far with probability
$\Omega(1/(d 2^{d}))$.
We give an asymptotically optimal analysis of this test, and show
that it rejects functions that are (even only) $\Omega(2^{-d})$-far
with $\Omega(1)$-probability (so the rejection probability
is a universal constant independent of $d$ and $n$).
This implies a tight relationship between
the $(d+1)^{\rm{st}}$-Gowers norm of a function and its maximal
correlation with degree $d$ polynomials, when the correlation is close to 1.


Our proof works by induction on $n$ and yields a new analysis of
even the
classical Blum-Luby-Rubinfeld~\cite{BLR} linearity test, for the
setting of functions mapping $\F_2^n$ to
$\F_2$. The optimality follows from a tighter analysis of
counterexamples to the ``inverse conjecture
for the Gowers norm'' constructed by \cite{GT07,LMS}.

Our result has several implications.
First, it shows that the Gowers norm test is tolerant, in that
it also accepts close codewords.
Second, it improves the
parameters of an XOR lemma for polynomials given by Viola and
Wigderson~\cite{VW}.
Third, it implies a ``query hierarchy''
result for property testing of affine-invariant properties.  That is, for every
function $q(n)$, it gives an affine-invariant property that is testable
with $O(q(n))$-queries, but not with $o(q(n))$-queries, complementing
an analogous  result of \cite{GKNR08} for graph properties.
\end{abstract}

\newpage

\ifnum\stoc=1
\setcounter{page}1
\fi

\section{Introduction}

Can the proximity of a function to a low-degree polynomial be estimated by sampling the function in few places?
Variants of this question have been studied in two different communities for different purposes.

\subsection{Gowers norm}

In the additive combinatorics community,
this issue arose in Gowers' notable improvement of Szemer\'edi's theorem, that any subset of the
integers with positive density has infinitely long arithmetic progressions.  To make his advance,
Gowers introduced his uniformity norms, now usually called Gowers norms.  The motivation for these
norms is that if a function $f$ has degree~$d$, then its derivative in direction $a$, $f(x+a)-f(x)$,
has degree at most $d-1$.  Hence the $(d+1)$-fold derivative is 0.  Thus, a natural test to decide
if a function $f$ has degree $d$ is to set $k=d+1$, evaluate the $k$-fold derivative of $f$ in $k$
random directions, and accept only if the derivative evaluates to 0.  This is what we call the
$k^{\rm{th}}$ Gowers norm test, $\tgowersk$, for $k=d+1$.

Our paper focuses on the field $\F_2$ of two elements, and we now restrict to this case.
The $k^{\rm{th}}$ Gowers norm of $f:\F_2 \to \F_2$, denoted $\| f \|_{U^k}$, is
given by the expression
$$ \| f \|_{U^k} \eqdef (\Pr[\hbox{$\tgowersk$ accepts}]  -
\Pr[\hbox{$\tgowersk$ rejects}])^{\frac{1}{2^k}}.$$

Gowers~\cite{Gow01} (see also \cite{GT05}) showed that the
correlation of $f$
to the closest degree $d$ polynomial is at most
$\| f \|_{U^{d+1}}$.
The well-known Inverse Conjecture for the Gowers Norm states that
some sort of converse holds:
if $\| f \|_{U^{d+1}} = \Omega(1)$, then the correlation of $f$ to
some degree $d$ polynomial is
$\Omega(1)$.
Lovett et al.\ \cite{LMS} and Green and Tao \cite{GT07} disproved this
conjecture as stated, but a modification of the conjecture remains open,
and was recently proven in high characteristic \cite{TZ,GTZ09a,GTZ10}.
These conjectures and the Gowers norms have been extremely influential.
For example, Green and Tao \cite{GT07} used the Gowers norms over the integers to prove that
the primes contain arbitrarily long arithmetic progressions.

Study of the Gowers norms over $\F_2$ has led to impressive results in theoretical computer science.
Samorodnitsky and Trevisan \cite{SamT} used Gowers norms to obtain very strong PCPs for Unique-Games-hard languages.
This implied that Maximum Independent Set in graphs of maximum degree~$\Delta$ could not be approximated within
$\Delta/\polylog(\Delta)$ under the Unique Games Conjecture.
Using Gowers norms, Bogdanov and Viola \cite{BogV} gave a pseudorandom generator fooling low-degree polynomials over $\F_2$.
They could only prove their result under the inverse conjecture for the Gowers norm, but later Lovett \cite{Lov} and Viola \cite{Vio} used related ideas to prove an unconditional result.
Finally, Viola and Wigderson \cite{VW} used Gowers norms to prove ``XOR'' lemmas for correlation to
low-degree polynomials and to low communication protocols.

\subsection{Local testing of Reed-Muller codes}

Traditionally the Gowers norm is used in what Green and Tao call the 1\% setting, where the correlation of a function to its closest low-degree polynomial is non-negligible but small.
The 99\% setting, where the correlation is close to 1, was addressed by Alon, Kaufman, Krivelevich, Litsyn, and Ron~\cite{AKKLR},
and is the focus of our work.
More precisely, Alon et al.\ considered the question
of testing if a Boolean function $f:\F_2^n \to \F_2$,
given by an oracle, is close to
a degree~$d$ multivariate polynomial.
They rediscovered a variation of the Gowers norm test, where all the derivative directions are
linearly independent, and showed that this test suffices
for that setting.
Thus, their analysis gave the only known relationship between the Gowers norm and the proximity to low-degree polynomials in the 99\% setting.

However, their analysis was not optimal.  In this work,
we give an improved, asymptotically optimal, analysis
of the Gowers norm test.  This gives a tight connection with the Gowers norm in the 99\% setting.
Before we elaborate, let us introduce our framework.

Our question is also called testing of Reed-Muller codes, which are codes based on low-degree polynomials.
The Reed-Muller codes are parameterized by two
parameters: $n$, the number of variables, and $d$, the
degree parameter.
The Reed-Muller codes consist of all functions
from $\F_2^n \to \F_2$ that are evaluations of polynomials
of degree at most $d$.
We use $\rmdn$ to denote this class, i.e.,
$\rmdn = \{f:\F_2^n\to\F_2 | \deg(f) \leq d\}$.

The proximity of functions is measured by the
(fractional Hamming) distance.
Specifically, for functions $f,g: \F_2^n\to\F_2$, we let
the {\em distance} between them, denoted by $\delta(f,g)$,
be the quantity
$\Pr_{x \gets_U \F_2^n} [f(x) \ne g(x)]$.
For a family of functions $\calf \subseteq \{g : \F_2^n \to \F_2\}$
let $\delta(f,\calf) = \min\{\delta(f,g) | g \in \calf\}$.
We say $f$ is $\delta$-close to $\calf$ if $\delta(f,\calf)
\leq \delta$ and $\delta$-far otherwise.

Let $\delta_d(f) = \delta(f,\rmdn)$ denote the distance of
$f$ to the class of degree $d$ polynomials.
The goal of Reed-Muller testing is to ``test'', with
``few queries'' of $f$, whether $f \in \rmdn$ or $f$ is far from
$\rmdn$.
Specifically, for a function $q:\Z^+ \times \Z^+ \times (0,1] \to \Z^+$,
a {\em $q$-query tester} for the class $\rmdn$ is
a randomized oracle algorithm $T$ that, given oracle access to some
function $f : \F_2^n \to \F_2$ and a proximity parameter $\delta \in
(0,1]$, queries at most
$q = q(d,n,\delta)$ values of $f$ and accepts $f \in \rmdn$ with
probability
$1$, while if $\delta(f,\rmdn)\geq \delta$ it rejects with probability
at least, say, $2/3$.
The function $q$ is the {\em query complexity} of the test and the
main goal
here is to minimize $q$, as a function possibly of
$d$, $n$ and $\delta$.  We denote the test $T$
run using oracle access to the function $f$ by $T^f$.

As mentioned earlier,
Alon et al.~\cite{AKKLR}
gave a tester with query complexity
$O(\frac d\delta \cdot 4^d )$.
Their tester consists of repetitions of a basic test, which we denote $\tgowers$.  $\tgowers$ is a
modification of the Gowers norm test $T_{{\rm
    GN}(d+1)}$ so that the $(d+1)$-fold derivatives are evaluated in $d+1$ random {\em linearly
  independent} directions.
This modified tester, whose rejection probability differs from that of the original Gowers norm
tester by only a constant factor,
can be described alternatively as follows.
Given oracle access
to $f$, $\tgowers$~selects a random $(d+1)$-dimensional affine subspace
$A$, and accepts if $f$ restricted
to $A$ is a degree $d$ polynomial. This requires $2^{d+1}$ queries
of $f$ (since that is the number of points contained in $A$).
Alon et al.\ show that if $\delta(f) \geq \delta$ then $\tgowers$
rejects $f$ with probability $\Omega(\delta/(d \cdot 2^{d}))$.
Their final tester then simply repeated $\tgowers$
$O(\frac d\delta \cdot 2^d)$ times and accepted if all invocations
of $\tgowers$ accepted.
The important
feature of this result is that the number of queries is independent
of $n$, the dimension of the ambient
space.
Alon et al.\ also show that any tester for $\rmdn$ must make at least
$\Omega(2^d+1/\delta)$ queries. Thus their result was tight to within
almost quadratic
factors, but left a gap open. We close this gap in this work.

\subsection{Main Result}

Our main result is an optimal analysis of the Gowers norm test, up to constants.
We show that if
$\delta_d(f) \geq 0.1$, in fact even if it's at least $0.1 \cdot 2^{-d}$,
then in fact the Gowers norm test rejects with probability lower bounded
by some {\em absolute constant}.
We now formally state our main theorem.
\begin{theorem}
\label{thm:main}
There exists a constant $\epsilon_1 > 0$ such
that for all $d, n$,
and for all functions $f:\F_2^n \to \F_2$,
we have\footnote{For a tester $T$ and a function $f$, the notation $T^f$ indicates the execution of
  $T$ with oracle access to $f$.} $$ \Pr[ \tgowers^f \mbox{ rejects}] \geq \min\{2^{d}\cdot \delta_d
(f), \epsilon_1\}.$$
\end{theorem}

Therefore, to reject
functions $\delta$-far from $\rmdn$ with constant probability, a tester
can repeat the test $\tgowers$
at most $O(1/\min\{2^{d}\delta_d(f), \epsilon_1\}) = O(1+\frac{1}{2^d
\delta})$ times,
making the total query complexity $O(2^d + 1/\delta)$.  This query
complexity is asymptotically tight in view of
the earlier mentioned lower bound in \cite{AKKLR}.

Our error-analysis is also asymptotically tight. Note that our theorem
effectively states that functions that are accepted by $\tgowers$ with
constant probability (close to 1) are (very highly) correlated
with degree $d$ polynomials. To get a qualitative improvement
one could hope that every function that is accepted by $\tgowers$
with probability strictly greater than half is somewhat correlated
with a degree d polynomial.
Such stronger statements however are effectively ruled out by
the counterexamples to the ``inverse conjecture for the Gowers norm''
given by \cite{LMS,GT07}.
Since the analysis given in these works does not match our parameters
asymptotically, we show (see Theorem~\ref{thm:counter}
in Appendix~\ref{app}) how an early analysis due to the authors of
\cite{LMS} can be used to show the asymptotic tightness of
the parameters of Theorem~\ref{thm:main}.

Our analysis of the Gowers norm test implies
a tight relationship between the Gowers norm and distance to degree~$d$
in the 99\% setting.
In particular, we show the following theorem.

\begin{theorem}\label{thm:tightgowers}
There exists $\epsilon>0$ such that
if $\| f \|_{U^{d+1}} \geq 1-\epsilon/2^d$, then $\delta_d(f) = \Theta
(1-\| f \|_{U^{d+1}})$.
\end{theorem}

For comparison, the best previous lower bound comes from the Alon et al.\ work,
whose
result can be
interpreted as showing that there exists $\epsilon > 0$ such that
if $\| f \|_{U^{d+1}} \geq 1-\epsilon/4^d$, then $\delta_d(f) = O(4^d
(1-\| f \|_{U^{d+1}}))$.

Before explaining our technique, we describe some applications of our result.

\subsection{Tolerant testing of RM codes}

Parnas, Ron, and Rubinfeld \cite{PRR} introduced the notion of tolerant testing, and
Guruswami and Rudra \cite{GR} studied this in the coding theoretic setting.
Standard testers are required to reject strings that are far from codewords, but are not required to accept strings that are close
to codewords.  A tolerant tester is required to accept close codewords.  In particular, for a code with minimum (relative) distance $\deltamin$,
there exists constants $c_1$ and $c_2$ such that the test must accept strings within distance $c_1 \deltamin$
with probability at least $2/3$ (called the acceptance condition),
and reject strings that are at least $(c_2 \deltamin)$-far with probability at least $2/3$ (called the rejection condition).

Any tester which satisfies the rejection condition must make at least $\Omega(1/\deltamin)$ queries.  We observe that a tester that satisfies the rejection condition and makes $C/\deltamin$ queries for a constant $C$ is also tolerant.  This follows because a string with distance $\deltamin/(3C)$ will be rejected with probability at most $1/3$.  It even suffices to have the rejection condition with a constant probability (instead of $2/3$), because the test can be repeated a constant number of times to boost the probability to $2/3$.

In particular,
for Reed-Muller codes $\deltamin = 2^{-d}$, so the Gowers norm test is also tolerant.
No tolerant tester for binary Reed-Muller codes appears to have been known.

\begin{theorem}
$\tgowers$ is a tolerant tester for $\rmdn$.
\end{theorem}

\subsection{XOR lemma for low-degree polynomials}

As mentioned earlier, Viola and Wigderson \cite{VW} used
the Gowers norm and the Alon et al.\ analysis
to give
an elegant ``hardness amplification'' result for low-degree
polynomials.
Let $f : \F_2^n \rightarrow \F_2$ be such that $\delta_d(f)$ is
noticeably large,
say $\geq 0.1$.  Viola and Wigderson showed how to use this $f$ to
construct a
$g : \F_2^m \rightarrow \F_2$ such that $\delta_d(g)$ is
significantly larger,
around $\frac{1}{2} - 2^{-\Omega(m)}$. In their construction,
$g=f^{\oplus t}$, the $t$-wise XOR of $f$, where $f^{\oplus t}:
(\F_2^n)^t \rightarrow \F_2$
is given by:
$$ f^{\oplus t} (x_1, \ldots, x_t) = \sum_{i = 1}^t f(x_i).$$
In particular, they showed that if $\delta_d(f) \geq 0.1$, then
$\delta_d(f^{\oplus t}) \geq 1/2 - 2^{-\Omega(t/4^d)}$.
Their proof proceeded by studying the rejection probabilities of $\tgowers$
on the functions $f$ and $f^{\oplus t}$. The analysis of the
rejection probability of $\tgowers$ given by \cite{AKKLR} was a central
ingredient in their proof.
By using our improved analysis of the rejection probability of $\tgowers$
from Theorem~\ref{thm:main} instead,
we get the following improvement.

\begin{theorem}
\label{thm:xor}
Let $\epsilon_1$ be as in Theorem~\ref{thm:main}.
Let $f : \F_2^n \rightarrow \F_2$.
Then
$$ \delta_d(f^{\oplus t}) \geq \frac{1 - (1 - 2 \min\{\epsilon_1/4, 2^
{d-2} \cdot \delta_d(f)\})^{t/2^d}}{2}.$$
In particular, if $\delta_d(f) \geq 0.1$, then $\delta_d(f^{\oplus
t}) \geq 1/2 - 2^{-\Omega(t/2^d)}.$
\end{theorem}

\subsection{Query hierarchy for affine-invariant properties}

Our result falls naturally in the general framework
of property testing~\cite{BLR,RS,GGR}.
Goldreich et al.~\cite{GKNR08} asked an
interesting question in this broad framework:
Given an
ensemble of properties $\calf = \{\calf_N\}_N$ where
$\calf_N$ is a property of functions on domains of size $N$,
which functions correspond to the query complexity of some property?
That is, for a given complexity function $q(N)$, is there
a corresponding property $\calf$ such that $\Theta(q(N))$-queries
are necessary and sufficient for testing membership in
$\calf_N$? This question is interesting even when we
restrict the class of properties being considered.

For completely general properties this question is
easy to solve. For graph properties \cite{GKNR08} et al.
show that for every efficiently
computable function $q(N) = O(N)$ there is a graph property
for which $\Theta(q(N))$ queries are necessary and sufficient
(on graphs on $\Omega(\sqrt{N})$ vertices).
Thus this gives a ``hierarchy theorem'' for query complexity.

Our main theorem settles the analogous question in the
setting of ``affine-invariant'' properties. Given a
field $\F$, a property $\calf \subseteq \{\F^n \to \F\}$
is said to be affine-invariant if for every $f \in \calf$
and affine map $A:\F^n \to \F^n$, the composition of
$f$ with $A$, i.e, the function
$f \circ A(x) = f(A(x))$, is also in $\calf$.
Affine-invariant properties seem to be the algebraic analog of
graph-theoretic properties and generalize most natural
algebraic properties (see Kaufman and Sudan~\cite{KS}).

Since the Reed-Muller codes form an affine-invariant family,
and since we have a tight analysis for their query complexity,
we can get the affine-invariant version of the result of
\cite{GKNR08}.
Specifically, given
any (reasonable) query complexity function $q(N)$
consider $N$ that is a power of two and
consider the class of functions on $n = \log_2 N$
variables of degree at most $d = \lceil \log_2 q(N) \rceil$.
We have that membership in this family requires
$\Omega(2^d) = \Omega(q(N))$-queries, and on the
other hand $O(2^d) = O(q(N))$-queries also suffice, giving
an ensemble of properties $\calp_N$ (one for every $N = 2^n$)
that is testable with $\Theta(q(N))$-queries.

\begin{theorem}
\label{thm:hier}
For every $q : \N \to \N$ that is at most linear, there is an
affine-invariant property
that is testable with $O(q(n))$ queries (with one-sided error) but is
not testable in
$o(q(n))$ queries (even with two-sided error).  Namely, this property
is membership
in $\RM({\lceil\log_2 q(n)\rceil, n})$.
\end{theorem}

\subsection{Technique}
\label{ssec:tech}

Our main theorem (Theorem~\ref{thm:main}) is obtained
by a novel proof that gives a (yet another!) new
analysis even of the classical linearity test of Blum, Luby,
Rubinfeld~\cite{BLR}.
The heart of our proof
is an inductive argument on $n$,
  the dimension of the ambient space.
While proofs that use induction on $n$ have been used before in the
literature
on low-degree testing (see, for instance, \cite{BFL,BFLS,FGLSS}),
they tend
to have a performance guarantee that degrades significantly with $n$.
Indeed no inductive proof was known even for the case of testing
linearity of functions from $\F_2^n \to \F_2$ that showed that functions
at $\Omega(1)$ distance from linear functions are rejected with
$\Omega(1)$ probability. (We note that the original analysis of
\cite{BLR} as well as the later analysis of \cite{BCHKS} do
give such bounds - but they do not use induction on $n$.)
In the process of giving a tight analysis of the \cite{AKKLR}
test for Reed-Muller codes, we thus end up giving a new (even if
weaker) analysis of the linearity test over $\F_2^n$. Below
we give the main idea behind our proof.

Consider a function $f$ that is $\delta$-far from every
degree $d$ polynomial. For a ``hyperplane'', i.e., an
$(n-1)$-dimensional affine subspace $A$ of $\F_2^n$, let
$f|_A$ denote the restriction of $f$ to $A$. We first
note that the test can be interpreted as first picking
a random hyperplane $A$ in $\F_2^n$ and then picking
a random $(d+1)$-dimensional affine subspace $A'$ within
$A$ and testing if $f|_{A'}$
is a degree $d$ polynomial. Now, if on every hyperplane $A$,
$f|_A$ is still $\delta$-far from degree $d$ polynomials
then we would be done by the inductive hypothesis.
In fact our hypothesis gets weaker as $n \to \infty$, so
that we can even afford a few hyperplanes where
$f|_A$ is not $\delta$-far. The crux of our analysis is
when for several (but just $O(2^d)$) hyperplanes $f|_A$ is close to some degree $d$ polynomial $P_A$. In this case
we manage to ``sew'' the different polynomials $P_A$ (each
defined on some $(n-1)$-dimensional subspace within $\F_2^n$)
into a degree $d$ polynomial $P$ that agrees with {\em all}
the $P_A$'s. We then show that this polynomial is close
to $f$, completing our argument.

To stress the novelty of our proof, note that this is not
a ``self-correction'' argument as in \cite{AKKLR},
where one defines a natural
function that is close to $P$, and then works hard to prove
it is a polynomial of appropriate degree. In contrast, our function
is a polynomial by construction and the harder part (if any)
is to show that the polynomial is close to $f$.
Moreover, unlike other inductive proofs, our main gain is
in the fact that the new polynomial $P$ has degree no greater
than that of the polynomials given by the induction.

\paragraph{Organization of this paper:}
We prove our main theorem,  Theorem~\ref{thm:main},
in Section~\ref{sec:mainproof} assuming three lemmas, two of which
study the rejection
probability of the $k$-dimensional affine subspace test, and another
that
relates the rejection probability of the basic $(d+1)$-dimensional
affine subspace test to that
of the $k$-dimensional affine subspace test.
\ifnum\stoc=1
The first two lemmas are proved in the following section,
Section~\ref{sec:analysis}, while the third is deferred to the
appendix. The appendix also includes details of the new
relationship to the Gowers norm,
the proof of our improved hardness amplification theorem,
(Theorem~\ref{thm:xor})
and the tightness analysis of our main theorem (based on a
tightening of a proof due to \cite{LMS}).
Some proofs that are abbreviated in this version may be
found in more detail in the full version of this paper~\cite{BKSSZ}.
\else
These three lemmas are
proved in the following section,
Section~\ref{sec:analysis}.

We give the relationship to the Gowers norm in Section~\ref
{sec:gowers}, and
we prove our improved hardness amplification theorem,
Theorem~\ref{thm:xor},
in Section~\ref{sec:xor}.
Finally, we show the tightness of our main theorem in the appendix.
\fi

\section{Proof of Main Theorem}
\label{sec:mainproof}

In this section we prove Theorem~\ref{thm:main}.
We start with an overview of our proof.
Recall that a $k$-flat is an affine subspace of
dimension $k$, and a hyperplane is an $(n-1)$-flat.

The proof of the main theorem proceeds as follows. We begin by
studying a variant of the basic tester $\tgowers$, which we call $T_{d,k}$
or the {\em $k$-flat test}. For an integer $k \geq d+1$,
$T_{d,k}^f$ picks a uniformly random $k$-flat in $\F_2^n$,
and accepts if and only if the restriction of $f$ to that flat has
degree at most~$d$.
In this language, the tester $\tgowers$ of interest to us is $T_{d,d+1}$.
To prove Theorem~\ref{thm:main}, we first show that for $k \approx d
+ 10$,
the tester $T^f_{d,k}$ rejects with constant probability if $\delta_d
(f)$ is $\Omega(2^{-d})$ (see Lemma~\ref{lem:induction}). We then
relate the rejection probabilities
of $T_{d,k}^f$ and $\tgowers^f$ (see Lemma~\ref{lem:ktod}).

The central ingredient in our analysis is thus Lemma~\ref
{lem:induction} which
is proved by induction on $n$, the dimension of the ambient space.
Recall that
we want to show that the two quantities (1) $\delta_d(f)$ and (2) $\Pr
[ T_{d,k}^f \mbox{ rejects}]$,
are closely related.
We consider what happens to $f$ when restricted to some
hyperplane $A$. Denote such a restriction
by $f|_A$.
For a hyperplane $A$ we consider the corresponding two quantities
(1) $\delta_d(f|_A)$ and (2) $\Pr[ T_{d,k}^{f|_A} \mbox{ rejects}]$.
The inductive hypothesis tells us that these two quantities are
closely related for
each $A$. Because of the local nature of tester $T_{d,k}$, it follows
easily that
$\Pr[ T_{d,k}^f \mbox{ rejects}]$ is the average of $\Pr[ T_{d,k}^{f|
_A} \mbox{ rejects}]$
over all hyperplanes $A$. The main technical content of
Lemma~\ref{lem:induction} is that there is a similar tight
relationship between $\delta_d(f)$
and the numbers $\delta_d(f|_A)$ as $A$ varies over all hyperplanes $A$.
This relationship suffices to complete the proof.
The heart of our analysis focuses on the case where for many
hyperplanes (about
$2^k$ of them, independent of $n$),
the quantity $\delta_d(f|_A)$ is very small (namely,
for many $A$, there is a polynomial $P_A$ of degree $d$ that is very
close to $f|_A$).
In this case, we show how to ``sew'' together the polynomials
$P_A$ to get a
polynomial $P$ on $\F_2^n$ that is also very close to $f$.
In contrast to prior approaches which yield a polynomial $P$
with larger degree
than that of the $P_A$'s, our analysis crucially
preserves this degree, leading to the eventual tightness of our
analysis.


We now turn to the formal proof.

\subsection{Preliminaries}

We begin by formally introducing the $k$-flat
test and some related notation.

\begin{definition}[$k$-flat test $T_{d,k}$]
The test $T_{d,k}^f$ picks a random $k$-flat $A \subseteq \F_2^n$
and accepts if and only if $f|_A$ ($f$ restricted to $A$) is a
polynomial
of degree at most $d$.

The rejection probability of $T_{d,k}^f$ is denoted $\Rej_{d, k}(f)$.
In words, this is the probability that $f|_A$ is not a degree $d$
polynomial when
$A$ is chosen uniformly at random among all $k$-flats
of $F_2^n$.
\end{definition}

Although we don't need it for our argument, we note that $\tgowers = T_{d,d+1}$
accepts if and only if the $2^{d+1}$ evaluations $f|_A$ sum to 0.

The following folklore proposition shows that for $k \geq d+1$, $T_
{d,k}$ has
perfect completeness.

\begin{proposition}
For every $k \geq d+1$,
$\delta_d(f) = 0$ if and only if $\Rej_{d,k}(f) = 0$.
\end{proposition}

\subsection{Key Lemmas}
We now state our three key lemmas, and then use them to
finish the proof
of Theorem~\ref{thm:main}.
The first is a simple lemma that says if the function
is sufficiently close to a degree $d$ polynomial,
then the rejection probability is linear in its distance
from degree $d$ polynomials.
\begin{lemma}
\label{lem:pw}
For every $k, \ell, d$ such that $k \geq \ell \geq d+1$, if $\delta
(f) = \delta$ then
$\Rej_{d,k}(f) \geq 2^\ell \cdot \delta \cdot (1 - (2^\ell-1)\delta)$.
In particular, if $\delta \leq 2^{-(d+2)}$ then
$\Rej_{d,k}(f) \geq \min\{\frac18,2^{k-1} \cdot \delta\}$.
\end{lemma}

The next lemma is the heart of our analysis and allows us to lower
bound the rejection
probability when the function is bounded away from degree $d$
polynomials.
\begin{lemma}
\label{lem:induction}
There exist positive constants $\beta < 1/4, \epsilon_0, \gamma$
and $c$ such that
the following holds for every $d, k, n$, such that $n \geq k \geq d+c$.
Let $f:\F_2^n \to \F_2$ be such that $\delta(f) \geq
\beta \cdot 2^{-d}$. Then $\Rej_{d,k}(f) \geq \epsilon_0 + \gamma
\cdot 2^d/2^n$.
\end{lemma}

The final lemma relates the rejection probabilities of different
dimensional tests.
\begin{lemma}
\label{lem:ktod}
For every $n, d$ and $k \geq k' \geq d+1$, and every $f: \F_2^n
\rightarrow \F_2$, we have
$$ \Rej_{d, k'}(f) \geq \Rej_{d, k}(f) \cdot 2^{-(k-k')}.$$
\end{lemma}

Given the three lemmas above, Theorem~\ref{thm:main} follows easily
as shown below.

\begin{proofof}{Theorem~\ref{thm:main}}
Let $\epsilon_0$ and $c$ be as in Lemma~\ref{lem:induction}.
We prove the theorem for $\epsilon_1 =
\epsilon_0 \cdot 2^{-(c-1)}$.
First note that if $\delta(f) \leq 2^{-(d+2)}$,
then we are done by Lemma~\ref{lem:pw}.
So assume $\delta(f) \geq 2^{-(d+2)} \geq \beta\cdot 2^{-d}$,
where $\beta$ is the constant from Lemma~\ref{lem:induction}.
By Lemma~\ref{lem:induction}, we know that $\Rej_{d,d+c}(f) \geq
\epsilon_0$.
Lemma~\ref{lem:ktod} now implies that
$\Rej_{d, d+1}(f) \geq \epsilon_0 \cdot 2^{-(c-1)}$, as desired.
\end{proofof}

\section{Analysis of the $k$-flat test}
\label{sec:analysis}

Throughout this section we fix $d$, so
we suppress it in the subscripts and
simply use $\delta(f) = \delta_d(f)$ and
$\Rej_k(f) = \Rej_{d,k}(f)$.

\subsection{Lemma \ref{lem:pw}: When $f$  is close to $\rmdn$}

Recall that we wish to prove

\noindent {\bf Lemma~\ref{lem:pw} (recalled):}
{\em For every $k, \ell, d$ such that $k \geq \ell \geq d+1$, if $
\delta(f) = \delta$ then
$\Rej_{k}(f) \geq 2^\ell \cdot \delta \cdot (1 - (2^\ell-1)\delta)$.
In particular, if $\delta \leq 2^{-(d+2)}$ then
$\Rej_{k}(f) \geq \min\{\frac18,2^{k-1} \cdot \delta\}$.
}

\begin{proofof}{Lemma \ref{lem:pw}}
The main idea is to show that with good probability, the flat will
contain exactly one point where
$f$ and the closest degree~$d$ polynomial differ, in which case the
test will reject.
The main claim we prove is that $\Rej_{\ell}(f) \geq
2^\ell \cdot \delta \cdot (1 - (2^\ell-1)\delta)$.
The first part then follows from the monotonicity of
the rejection probability, i.e., $\Rej_k(f) \geq
\Rej_{\ell}(f)$ if $k \geq \ell$. The second part follows
by setting $\ell = k$ if $\delta \leq 2^{-(k+1)}$ and
$\ell$ such that $2^{-(\ell+2)} < \delta \leq 2^{-(\ell+1)}$
otherwise. In the former case, we get $\Rej_{k}(f) \geq
2^{-(k-1)}\cdot\delta$ while in the latter case we get
$\Rej_{k}(f) \geq \Rej_{\ell}(f) \geq \frac18$.
We thus turn to proving
$\Rej_{\ell}(f) \geq
2^\ell \cdot \delta \cdot (1 - (2^\ell-1)\delta)$.

Let $g\in \rmdn$ be a polynomial achieving
$\delta(f) = \delta(f,g)$. Consider a random $\ell$-flat
$A$ of $\F_2^n$. We think of the points of $A$ as generated
by picking a random full-rank matrix $M \in \F_2^{n\times \ell}$ and a random
vector
$b \in \F_2^n$, and then letting $A = \{ a_x \eqdef M x + b \mid x
\in \F_2^{\ell}\}.$
Thus the points of $A$ are indexed by elements of $\F_2^\ell$.

For $x \in \F_2^\ell$, let $E_x$ be
the event that ``$f(a_x) \ne g(a_x)$''.
Further let $F_x$ be the event that ``$f(a_x) \ne g(a_x)$
and $f(a_y) = g(a_y)$ for every $y \ne x$''.
We note that if any of the events $F_x$ occurs (for $x \in
\F_2^\ell$), then
the $\ell$-flat test rejects $f$.
This is because distinct degree $d$ polynomials differ in at least
$2^{-d}$ fraction of points, so they cannot differ in exactly one point
if $\ell > d$.

We now lower bound the probability
of $\cup_x F_x$.
Using the fact that
$a_x$ is distributed
uniformly over $\F_2^n$ and
$a_y$ is distributed uniformly over $\F_2^n - \{a_x\}$, we note that
$\Pr[E_x] = \delta$ and $\Pr[E_x \mbox{ and } E_y] \leq \delta^2$.
We also have $\Pr[F_x] \geq \Pr[E_x] - \sum_{y \ne x} \Pr[E_x
\mbox{ and } E_y] \geq \delta - (2^\ell - 1)\cdot \delta^2$.
Finally, noticing that the events $F_x$ are mutually exclusive
we have that $\Pr[\cup_i F_i] = \sum_i \Pr[F_i] \geq
2^\ell \cdot \delta \cdot (1 - (2^\ell - 1)\cdot \delta)$,
as claimed.
\end{proofof}

\subsection{Lemma \ref{lem:induction}: When $f$ is bounded away from $
\rmdn$}

The main idea of the proof of Lemma \ref{lem:induction} is to
consider the restrictions
of $f$ on randomly chosen ``hyperplanes'', i.e., $(n-1)$-flats. If on
an overwhelmingly large fraction (which will be
quantified in the proof) of hyperplanes, our function
is far from degree $d$ polynomials, then the inductive hypothesis
suffices to show that $f$ will be rejected with high probability
(by the $k$-flat test). The interesting case is
when the restrictions of $f$ to several hyperplanes are close to
degree $d$ polynomials. In Lemma~\ref{lem:patch} we
use the close polynomials on such hyperplanes to construct a
polynomial that
has significant agreement with $f$ on the union of the hyperplanes.

We start by first fixing some terminology.
We say $A$ and $B$ are {\em complementary}
hyperplanes if $A \cup B = \F_2^n$. Recalling that
a hyperplane is the set of points $\{x \in \F_2^n | L(x) = b\}$
where $L: \F_2^n \rightarrow \F_2$ is a nonzero linear function and $b
\in \F_2$, we refer to $L$
as the linear part of the hyperplane. We say that hyperplanes $A_1,
\ldots,A_{\ell}$
are linearly independent if the corresponding linear parts are
independent.
The following proposition lists some basic facts about hyperplanes
that we
use. The proof is omitted.

\begin{proposition}[Properties of hyperplanes]

\label{prop:hyperplane}
\begin{enumerate}
\item There are exactly $2^{n+1} - 2$ distinct hyperplanes in $\F_2^n$.
\item Among any $2^{\ell}-1$ distinct hyperplanes, there are
at least $\ell$ independent hyperplanes.
\item There is an affine invertible transform that maps independent
hyperplanes $A_1,\ldots,A_{\ell}$ to
the hyperplanes $x_1 = 0, x_2 = 0, \ldots, x_{\ell} = 0$.
\end{enumerate}
\end{proposition}

We are now ready to prove Lemma~\ref{lem:induction}.
We first recall the statement.

\noindent {\bf Lemma~\ref{lem:induction} (recalled):}
{\em There exist positive constants $\beta < 1/4, \epsilon_0, \gamma$
and $c$ such that
the following holds for every $d, k, n$, such that $n \geq k \geq d+c$.
Let $f:\F_2^n \to \F_2$ be such that $\delta(f) \geq
\beta \cdot 2^{-d}$. Then $\Rej_{d,k}(f) \geq \epsilon_0 + \gamma
\cdot 2^d/2^n$.
}

\begin{proofof}{Lemma~\ref{lem:induction}}
We prove the lemma for every
$\beta < 1/24$,
$\epsilon_0 < 1/8$,
$\gamma \geq 72$, and $c$ such that
$2^c \geq \max\{4\gamma/(1 - 8\epsilon_0),\gamma/(1-\epsilon_0),2/
\beta\}$.
(In particular, the choices $\beta = 1/25$, $\epsilon_0 = 1/16$, $
\gamma = 72$
and $c = 10$ work.)

The proof uses
induction on $n-k$.
When $n = k$ we have $\Rej_{k}(f) = 1 \geq
\epsilon_0 + \gamma \cdot 2^{d-k}$ as required, because $2^c \geq
\frac{\gamma}{1-\epsilon_0}$.   So we move
to the inductive step.

Let $\mathcal H$ denote the set of hyperplanes in $\F_2^n$.
Let $N = 2(2^n-1)$ be the cardinality of $\mathcal H$.
Let $\mathcal H^*$ be the set of all the hyperplanes $A \in \mathcal H$
such that $\delta(f|_{A},\RM(d,n-1)) < \beta \cdot 2^{-d}$.
Let $K = |\mathcal H^*|$.

Now because a random $k$-flat of a random hyperplane {\em is} a random
$k$-flat, we have
\ifnum\stoc=0
$$\Rej_k(f) = \mathbb E_{A \in \mathcal H}[ \Rej_{k}(f |_A) ].$$
\else
$\Rej_k(f) = \mathbb E_{A \in \mathcal H}[ \Rej_{k}(f |_A) ]$.
\fi
By the induction hypothesis, for any $A \in \mathcal H \setminus
\mathcal H^*$, we have
\ifnum\stoc=0
$$\Rej_k(f |_A) \geq \epsilon_0 + \gamma \cdot \frac{2^d}{2^{n-1}}.$$
\else
$\Rej_k(f |_A) \geq \epsilon_0 + \gamma \cdot \frac{2^d}{2^{n-1}}$.
\fi
Thus,
\ifnum\stoc=0
$$\Rej_{k}(f) \geq \epsilon_0 + \gamma\cdot \frac{2^d}{2^{n-1}} - K/N.$$
\else
$\Rej_{k}(f) \geq \epsilon_0 + \gamma\cdot \frac{2^d}{2^{n-1}} - K/N$.
\fi
We now take cases on whether $K$ is large or small:

\begin{enumerate}
\item {\bf Case 1:} $K \leq \gamma \cdot 2^d$.\\
In this case,
$\Rej_{k}(f) \geq \epsilon_0 + \gamma \cdot 2^d/2^{n-1} - K/N \geq
\epsilon_0 + \gamma\cdot 2^{d}/2^{n}$ as desired.
\item {\bf Case 2:} $K > \gamma \cdot 2^d$.\\
Lemma~\ref{lem:patch} (below) shows that in this case,
$\delta(f) \leq \frac{3}{2}\beta \cdot 2^{-d} + 9/(\gamma 2^d)
\eqdef \delta_0$,
provided $\beta\cdot 2^{-d} < 2^{-(d+2)}$
(which holds since $\beta < 1/24 < 1/4$).
\ifnum\stoc=1
From our choice of parameters, this turns out to be a sufficient
bound on $\delta(f)$ for the lemma.
\else

Note that since
$\beta < 1/24$ and $9/\gamma < 1/8$,
we get $\delta_0 < 2^{-(d+2)}$ and so
Lemma~\ref{lem:pw} implies
that $\Rej_{k}(f) \geq \min\{2^{k-1} \cdot
\delta(f),\frac18\}
\geq \min\{2^{k-1} \cdot \beta \cdot 2^{-d},\frac18\}$.
We verify both quantities above are at least
$\epsilon_0 + \gamma/2^{(c+1)} \geq \epsilon_0 + \gamma 2^d/2^n$.
The condition $1/8 > \epsilon_0 + \gamma/2^{c+1}$ follows
from the fact that $2^c \geq 4\gamma/(1 - 8\epsilon_0)$.
To verify the second condition, note that
$2^{k-1}\cdot\beta\cdot 2^{-d}
\geq 2^{c-1}\beta \geq 1$ since $2^c \geq 2/\beta$.
\fi

\end{enumerate}

\ifnum\stoc=0
We thus conclude that the rejection probability of $f$
is at least $\epsilon_0 + \gamma\cdot 2^d/2^n$ as claimed.
\fi
\end{proofof}

\begin{lemma}
\label{lem:patch}
For $f : \F_2^n \to \F_2$, let $A_1,\ldots,A_K$ be
hyperplanes such that $f|_{A_i}$ is $\alpha$-close
to some degree $d$ polynomial on $A_i$. If $K > 2^{d+1}$ and
$\alpha < 2^{-(d+2)}$, then
$\delta(f) \leq \frac{3}{2}\alpha + 9/K$.
\end{lemma}

\begin{proof}
Let $P_i$ be the degree $d$ polynomial such that
$f|_{A_i}$ is $\alpha$-close to $P_i$.

\begin{claim}
\label{clm:intersect}
If $4 \alpha < 2^{-d}$ then for every pair of
hyperplanes $A_i$ and $A_j$, we have $P_i|_{A_i \cap A_j}
= P_j|_{A_j \cap A_i}$.
\end{claim}

\begin{proof}
If $A_i$ and $A_j$ are complementary then this is vacuously
true. Otherwise, $|A_i \cap A_j| = |A_i|/2 = |A_j|/2$.
So $\delta(f|_{A_i \cap A_j}, P_i|_{A_i \cap A_j}) \leq
2 \delta(f|_{A_i},P_i) \leq 2\alpha$ and similarly
$\delta(f|_{A_i \cap A_j}, P_j|_{A_i \cap A_j}) \leq
2\alpha$. So $\delta(P_i|_{A_i \cap A_j}, P_j|_{A_i \cap A_j})
\leq 4 \alpha < 2^{-d}$. But these are both degree $d$
polynomials and so if their proximity is less than $2^{-d}$
then they must be identical.
\end{proof}

Let $\ell = \lfloor \log_2 (K+1) \rfloor$. Thus $\ell > d$.
By Proposition~\ref{prop:hyperplane} there are at least
$\ell$ linearly independent hyperplanes
among $A_1,\ldots,A_K$. Without loss of generality let these
be $A_1,\ldots,A_{\ell}$. Furthermore,
by an affine transformation of coordinates, for $i \in [\ell]$
let $A_i$ be the hyperplane $\{ x \in \F_2^n \mid x_i = 0\}$.
For $i \in [\ell]$ extend $P_i$ to a function
on all of $\F_2^n$ by making $P_i$ independent
of $x_i$.
We will sew together $P_1,\ldots,P_\ell$ to get a polynomial close to $f$.

Let us write all functions from $\F_2^n \to \F_2$
as polynomials in $n$ variables $x_1,\ldots,x_{\ell}$
and $\vec{y}$ where $\vec{y}$ denotes the last $n-\ell$
variables. For $i \in [\ell]$ and $S \subseteq [\ell]$,
let $P_{i,S}(\vec{y})$ be the monomials of $P_i$ which
contain $x_i$ for $i \in S$, and no $x_j$ for $j \not\in [\ell]$.
That is, $P_{i,S}(\vec{y})$ are
polynomials such that
$P_i(x_1,\ldots,x_\ell,\vec{y}) =
\sum_{S \subseteq [\ell]} P_{i,S}(\vec{y}) \prod_{j \in S} x_j$.
Note that the degree of $P_{i,S}$ is at most $d - |S|$. (In particular,
if $|S| > d$, then $P_{i,S} = 0$.)
Note further that since $P_i$ is independent of $x_i$,
we have that $P_{i,S} = 0$ if $i \in S$.

\begin{claim}
\label{clm:agree}
For every $S \subseteq [\ell]$ and every $i,j \in [\ell] - S$, $P_
{i,S}(\vec{y}) = P_{j,S}(\vec{y})$.
\end{claim}

\begin{proof}
Note that
$P_i|_{A_i \cap A_j}(\vec x, \vec{y})
= \sum_{S \subseteq [\ell] - \{i,j\}}
P_{i,S}(\vec y) \prod_{m \in S} x_m$.
Similarly
$P_j|_{A_i \cap A_j}(\vec x, \vec{y})
= \sum_{S \subseteq [\ell] - \{i,j\}}
P_{j,S}(\vec y) \prod_{m \in S} x_m$.
Since the two functions are equal (by Claim~\ref{clm:intersect}),
we have that every pair of coefficients of $\prod_{m \in S} x_m$
must be the same. We conclude that $P_{i,S} = P_{j,S}$.
\end{proof}

Claim~\ref{clm:agree} above now allows us to define,
for every $S \subsetneq [\ell]$, the polynomial $P_S(\vec{y})$
as the unique polynomial $P_{i,S}$ where $i \not\in S$.
We define
$$P(x_1,\ldots,x_{\ell},\vec{y}) = \sum_{S \subsetneq [\ell]}
P_S(\vec{y}) \prod_{j \in S} x_j.$$
By construction, the degree of $P$ is at most $d$. This is the
polynomial that
we will eventually show is close to $f$.

\begin{claim}
For every $i \in [K]$,
$P|_{A_i} = P_i|_{A_i}$.
\end{claim}

\begin{proof}
First note that for each $i \in [\ell]$, $P|_{A_i} = P_i|_{A_i}$.
This is because the coefficients of the two polynomials become identical
after substituting $x_i = 0$ (recall that $A_i$ is the hyperplane
$\{x \in \F_2^n \mid x_i = 0\}$).

Now 
consider general $i \in [K]$.
For any point $x \in A_i \cap (\bigcup_{j=1}^\ell A_j)$,
letting $j^* \in [\ell]$ be such that $x \in A_{j^*}$, we have $P_i
(x) = P_{j*}(x)$ (by Claim~\ref{clm:intersect})
and $P_{j^*}(x) = P(x)$ (by what we just showed, since $j^* \in [\ell]
$).
Thus $P$ and $P_i$ agree on all points in $A_i \cap (\bigcup_{j=1}^
\ell A_j)$.
Now since $\ell > d$, we have that $|A_i \cap (\bigcup_{j=1}^\ell
A_j)|/|A_i| \geq 1 - 2^{-\ell} > 1 - 2^{-d}$,
and since $P|_{A_i}$ and $P_i|_{A_i}$ are both degree $d$
polynomials, we conclude that
$P|_{A_i}$ and $P_i|_{A_i}$ are identical.
Thus for all $i \in [K]$, $P|_{A_i}= P_i|_{A_i}$.
\end{proof}

We will show below that $P$ is close to $f$,
by considering all the hyperplanes $A_1,\ldots,A_K$.
If these hyperplanes uniformly covered $F_2^n$, then we could
conclude $\delta(f,P) \leq \alpha$,
as $f$ is $\alpha$-close to $P$ on each hyperplane.  Since the $A_i$
don't uniformly cover $\F_2^n$,
we'll argue that almost all points are covered approximately the
right number of times, which will be
good enough.  To this end, let
\[ \bad = \set{z \in \F_2^n | \hbox{$z$ is contained in less
than $K/3$ of the hyperplanes $A_1,\ldots,A_K$}}.\]
Let $\tau=|\bad|/2^n$.

\begin{claim}
$\delta(f,P) \leq 3/2 \cdot \alpha + \tau$.
\end{claim}

\begin{proof}
Consider the following experiment: Pick $z \in \F_2^n$
and $i \in [K]$ uniformly and independently at random
and consider the probability that ``$z \in A_i$ and $f(z)
\ne P_i(z)$''.
\ifnum\stoc=1  
A simple analysis, using the fact that $P|_{A_i} = P_i$, shows
that
$$
\frac13 \cdot (\delta(f,P) - \tau) \leq
\Pr_{z,i}[z \in A_i ~\&~ f(z) \ne P_i(z) ]
 \leq  \frac{1}{2} \cdot \alpha.$$
Details omitted.
\else         
On the one hand, we have
\begin{eqnarray*}
\lefteqn{\Pr_{z,i}[z \in A_i ~\&~ f(z) \ne P_i(z) ] } \\
&\leq & \max_i \Pr_z [ z \in A_i] \cdot \Pr_z [f(z) \ne P_i(z) | z
\in A_i]\\
& \leq & \frac{1}{2} \cdot \alpha
\end{eqnarray*}

On the other hand, using the fact that $P|_{A_i} = P_i$, we have
that

\begin{eqnarray*}
\lefteqn{\Pr_{z,i}[z \in A_i ~\&~ f(z) \ne P_i(z) ] } \\
& =  & \Pr_{z,i}[z \in A_i ~\&~ f(z) \ne P(z) ]  \\
& \geq  & \Pr_{z,i}[z \in A_i ~\&~ f(z) \ne P(z) ~\&~ z \not\in
\bad]  \\
& = & \Pr_{z}[f(z) \ne P(z) ~\&~ z \not\in \bad] \cdot
\Pr_{z,i} [z \in A_i | f(z) \ne P(z) ~\&~ z \not\in \bad ]  \\
& \geq & \Pr_{z}[f(z) \ne P(z) ~\&~ z \not\in \bad] \cdot
\min_{z: z \not\in \bad~\&~f(z) \ne P(z)} \Pr_{i} [z \in A_i]  \\
& \geq & (\delta(f,P) - \tau) \cdot
\min_{z: z \not\in \bad} \Pr_{i} [z \in A_i ]  \\
& \geq & (\delta(f,P) - \tau) \cdot \frac13
\end{eqnarray*}

We thus conclude that $(\delta(f,P) - \tau)/3 \leq \alpha/2$
yielding the claim.
\fi             
\end{proof}

\begin{claim}
$\tau \leq 9/K$.
\end{claim}

\begin{proof}
The proof is a straightforward ``pairwise independence''
argument, with a slight technicality to handle complementary
hyperplanes.

Consider a random variable $z$ distributed uniformly
over $\F_2^n$.  For $i \in [K]$, let $Y_i$ denote the random variable
that is $+1$ if $z \in A_i$ and $-1$ otherwise.
Note that $z \in \bad$ if and only if $\sum_{i} Y_i \leq - K/3$
and so $\tau = \Pr [ \sum_i Y_i \leq - K/3]$.
We now bound this probability.

For every $i$, note that $E[Y_i] = 0$ and $\Var[Y_i] = 1$.
Notice further that if $A_i$ and $A_j$ are not complementary
hyperplanes, then $Y_i$ and $Y_j$ are independent and
so $E[Y_i Y_j] = 0$, while if they are complementary,
then $E[Y_i Y_j] = -1 \leq 0$.
We conclude that $E[\sum_i Y_i] = 0$ and
$\Var[\sum_i Y_i] \leq K$.
Using Chebychev's bound, we conclude
that $\tau = \Pr [\sum_i Y_i \leq - K/3] \leq
\Var(\sum_i Z_i)/ (K^2/9) \leq 9/K$.
\end{proof}

The lemma follows from the last two claims above.
\end{proof}

\newcommand{\diffdimsection}{
\ifnum\stoc=0
\subsection{Lemma \ref{lem:ktod}: Relating different dimensional tests}
\else
\section{Lemma \ref{lem:ktod}: Relating different dimensional tests}
\fi

\begin{lemma}
\label{lem:basecase}
Let $k \geq d+1$ and
let $f:\F_2^{k+1} \to \F_2$ have
degree greater than $d$.
Then $\Rej_{d,k}(f) \geq 1/2$.
\end{lemma}

\begin{proof}
Assume for contradiction that there is a strict majority of
hyperplanes $A$ on which $f|_A$ has degree
$d$. Then there exists two complementary hyperplanes
$A$ and $\bar{A}$ such that $f|_A$ and $f|_{\bar{A}}$
both have degree $d$. We can
interpolate a polynomial $P$ of degree at most
$d+1$
that now equals $f$ everywhere. If
$P$ is of degree $d$, we are done, so
assume $P$ has degree exactly $d+1$ and
let $P_h$ be the homogenous degree $d+1$ part
of $P$ (i.e,, $P = P_h + Q$ where $\deg(Q) \leq d$
and $P_h$ is homogenous).
Now consider all hyperplanes~$A$ such that $f|_A= P|_A$
has degree at most $d$. Since these form a strict majority,
there are at least
$\frac12 (2^{k+2} - 2) + 1 > 2^{k+1}-1$ such hyperplanes.
It follows that there are at least
$k+1 \geq d+2$ linearly independent hyperplanes such that
this condition holds. By an affine transformation
we can assume these hyperplanes are of the
form $x_1 = 0, \ldots, x_{d+2} = 0$.
But then $\prod_{i=1}^{d+2} x_i$ divides
$P_h$ which contradicts the fact that the degree of
$P_h$ is at most $d+1$.
\end{proof}

\begin{lemma}
\label{lem:exp}
Let $n \geq k \geq d+1$ and
let $f:\F_2^n \to \F_2$ have
degree greater than $d$.
Then $\Rej_{d,k}(f) \geq 2^{k-n}$.
\end{lemma}

\begin{proof}
The proof is a simple induction on $n$.
The base case of $n=k$ is trivial.  Now assume for $n-1$.
Pick a random hyperplane $A$. With
probability at least $1/2$ (by the previous
lemma), $f|_A$ is not a degree $d$ polynomial.
By the inductive hypothesis,
a random $k$-flat of $A$ will now detect that $f|_A$ is
not
of degree $d$ with probability $2^{k-n+1}$.
We conclude that a random $k$-flat of $\F_2^n$ yields a function of
degree greater than $d$ with probability at
least $2^{k-n}$.
\end{proof}


We now have all the pieces needed to prove Lemma \ref{lem:ktod}.

\noindent {\bf Lemma~\ref{lem:ktod} (recalled):}
{\em For every $n, d$ and $k \geq k' \geq d+1$, and every $f: \F_2^n
\rightarrow \F_2$, we have
$$ \Rej_{d, k'}(f) \geq \Rej_{d, k}(f) \cdot 2^{-(k-k')}.$$}

\begin{proofof}{Lemma~\ref{lem:ktod}}
We view the $k'$-flat test as the following process:
first pick a random $k$-flat $A_1$ of $\F_2^n$, then pick a
random $k'$-flat $A$ of $A_1$, and accept iff $f|_A$ is
a degree $d$ polynomial. Note that this is completely
equivalent to the $k'$-flat test.

To analyze our test, we first consider the event that
$f|_{A_1}$ is not a degree $d$ polynomial.
The probability that this happens is $\Rej_{d, k}(f)$.
Now conditioned on the event that
$f|_{A_1}$ is not a degree $d$ polynomial, we can now use
Lemma~\ref{lem:exp} to conclude that the probability that
$(f|_{A_1})|_A$ is not a degree $d$ polynomial is at least
$2^{-(k-k')}$.
We conclude that $\Rej_{d, k'}(f) \geq \Rej_{d,k}(f) \cdot 2^{-(k-k')}$.
The lemma follows.
\end{proofof}
}

\ifnum\stoc=0
\diffdimsection
\fi

\newcommand{\gowerssection}{
\section{Gowers norms}
\label{sec:gowers}

Our main theorem can be interpreted as giving a tight
relationship between the Gowers norm of a function $f$
and its proximity to some low degree polynomial. In this
section we describe this relationship.

We start by recalling the definition of the test $\tgowersk^{f}$ and the
Gowers norm $\| f \|_{U^k}$. On oracle access to function $f$,
the test $\tgowersk$ picks $x_0$ and directions $a_1,\ldots,a_k$
uniformly and independently in $\F_2^n$ and accepts if and only
if $f|_A$ is a degree $k-1$ polynomial, where $A = \{x_0 +
\mspan(a_1,\ldots,a_k)\}$.
The Gowers norm is given by the expression
$$ \| f \|_{U^k} \eqdef (\Pr[\hbox{$\tgowersk^{f}$ accepts}]  -
\Pr[\hbox{$\tgowersk^f$ rejects}])^{\frac{1}{2^k}}.$$

Our main quantity of interest is the correlation of
$f$ with degree $d$ polynomials, i.e., the quantity
$1 - 2\delta_d(f)$.

Our theorem relating the Gowers norm to the correlation
is given below.

\begin{theorem}\label{thm:tightgowers}
There exists $\epsilon>0$ such that
if $\| f \|_{U^{d+1}} \geq 1-\epsilon/2^d$, then $\delta_d(f) = \Theta
(1-\| f \|_{U^{d+1}})$.
\end{theorem}

To prove the theorem we first relate the rejection
probability of the test $T_{{\rm GN}(d+1)}$ with that of the test
$\tgowers$.

\begin{proposition}
\label{prop:relatetests}
For every $n \geq d+1$ and for every $f$,
$\Pr [T_{{\rm GN}(d+1)}^f {\rm rejects}] \geq \frac14 \cdot
\Pr [\tgowers^f{\rm rejects}]$.
\end{proposition}

\begin{proof}
We show that with probability at least 1/4, the $a_i$ are linearly
independent.
Consider picking $d$ independent vectors
$a_1,\ldots,a_d$
in $\F_2^n$. For fixed $\beta_1,\ldots,\beta_d \in \F_2$
(not all zero), the probability that $\sum_i \beta_i a_i = 0$
is at most $2^{-n}$. Taking the union bound over all sequences
$\beta_1,\ldots,\beta_d$ we find that the probability that
$a_1,\ldots,a_d$ have a linear dependency is at most $2^{d-n}
\geq \frac12$ if $n \geq d+1$. For any fixed $a_1,\ldots,a_d$,
the probability that $a_{d+1} \in \mspan(a_1,\ldots,a_d)$ is also
at most $\frac12$. Thus we find with probability at least
$1/4$, the vectors $a_1,\ldots,a_{d+1}$ are linearly independent
provided $n \geq d+1$.
The proposition follows since the rejection probability of
$T_{{\rm GN}(d+1)}^f$ equals the rejection probability of $\tgowers^f$ times
the probability that $a_1,\ldots,a_{d+1}$ are linearly
independent.
\end{proof}

We are now ready to prove Theorem~\ref{thm:tightgowers}.

\begin{proofof}{Theorem~\ref{thm:tightgowers}}
The proof is straightforward given our main theorem and
the work of Gowers et al.~\cite{Gow01, GT05}.
As mentioned earlier, Gowers already showed that
$1-2\delta_d(f) \leq \| f \|_{U^{d+1}}$ \cite{Gow01, GT05}, i.e.,
$\delta_d(f) \geq (1-\| f \|_{U^{d+1}})/2$.

For the other direction, suppose $\| f \|_{U^{d+1}} = 1-\gamma$,
where $\gamma \leq \epsilon/2^d$ for small enough $\epsilon$.
Let $\rho$ denote the rejection probability of $T_{{\rm GN}(d+1)}^f$.
By Proposition~\ref{prop:relatetests} we have $\rho \geq \frac14\cdot
\Rej_{d,d+1}(f)$.
By choosing $\epsilon$ small enough, we also have
$1-2\rho = \| f \|_{U^{d+1}}^{2^{d+1}} > 1-\epsilon_1/2$, i.e.,
$\rho < \epsilon_1/4$, so $\Rej_{d,d+1}(f) < \epsilon_1$.
Thus, by Theorem~\ref{thm:main},
\begin{eqnarray*}
\delta_d(f) &\leq& \frac{1}{2^d}\Rej_{d,d+1}(f) (f)\\
&\leq& \frac{1}{2^{d-2}}\rho\\
&=& \frac{1}{2^{d-1}} (1-\| f \|_{U^{d+1}}^{2^{d+1}})\\
&=& \frac{1}{2^{d-1}} (1-(1-\gamma)^{2^{d+1}})\\
&\leq& \frac{1}{2^{d-1}} (1-(1-O(2^{d+1} \gamma)))\\
&=& O(\gamma),
\end{eqnarray*}
as required.
\end{proofof}

}

\ifnum\stoc=0
\gowerssection
\fi

\newcommand{\xorsection}{%
\section{\label{sec:apps}XOR lemma for low-degree polynomials}
\label{sec:xor}
A crucial feature of the test $\tgowersk$
(that is not
a feature of the $k$-flat test for $k > d+1$) is
that the rejection probability of $f^{\oplus t}$ can be exactly
expressed as a
rapidly growing (in $t$) function of the rejection probability of $f$.
Let $\Rej^0_d(f)$ denote the rejection probability of $T_{{\rm GN}(d+1)}^f$.
Then we have:

\begin{proposition}
\label{prop:power}
$$(1- 2\Rej^0_d(f^{\oplus t})) = (1 - 2 \Rej^0_d(f) )^t.$$
\end{proposition}
\begin{proof}
We first note that the proposition is equivalent to showing
that $\| f^{\oplus t} \|_{U^{d+1}} =
\left(\| f \|_{U^{d+1}}\right)^t$.
It is a standard fact (e.g., Fact 2.6 in \cite{VW}) that for
functions $f, g$
on disjoint sets of inputs, $\| f(x) + g(y) \|_{U^{d+1}} =
\| f(x) \|_{U^{d+1}} \cdot \| g(y) \|_{U^{d+1}}$.
This immediately yields the proposition.
\end{proof}

We also use the following well-known relationship between
the Gowers norm and the correlation of a function to the
class of degree $d$ polynomials. (We state it in terms of
the rejection probability of the test $T_{{\rm GN}(d+1)}$.)

\begin{lemma}[\cite{Gow01, GT05}]
\label{lem:gowers}
$$1 - 2\delta_d(g) \leq (1- 2 \Rej^0_d(g))^{\frac{1}{2^d}}.$$
\end{lemma}

We are now ready to prove Theorem~\ref{thm:xor} which we recall below.

\noindent {\bf Theorem~\ref{thm:xor} (recalled):}
{\em
Let $f : \F_2^n \rightarrow \F_2$.
Then
$$ \delta_d(f^{\oplus t}) \geq \frac{1 - (1 - 2 \min\{\epsilon_1/4, 2^
{d-2} \cdot \delta_d(f)\})^{t/2^d}}{2}.$$
In particular, if $\delta_d(f) \geq 0.1$, then $\delta_d(f^{\oplus
t}) \geq \frac{ 1- 2^{-\Omega(t/2^d)}}{2}.$
}

\begin{proofof}{Theorem~\ref{thm:xor}}
By Theorem~\ref{thm:main} and Proposition~\ref{prop:relatetests},
$$\Rej^0_d(f) \geq \min\{ \epsilon_1/4, 2^{d-2} \cdot \delta_d(f) \}.$$
Thus by Proposition~\ref{prop:power},
$$(1 - 2\Rej^0_d(f^{\oplus t}))^{\frac{1}{2^d}} =
(1 - 2\Rej^0_d(f))^{\frac{t}{2^d}}  \leq (1 - 2 \min\{\epsilon_1/4,
2^{d-2} \cdot \delta_d(f)\})^{\frac{t}{2^d}}.$$
Finally, Lemma~\ref{lem:gowers} shows that
$$\delta_d(f^{\oplus t}) \geq \frac{1 - \left(
     1 - 2\min\{\epsilon_1/4, 2^{d-2} \cdot \delta_d(f)\}
     \right)^{\frac{t}{2^d}}}{2}.$$
\end{proofof}
}

\ifnum\stoc=0
\xorsection
\fi

\section*{Acknowledgments}

Thanks to Alex Samorodnitsky and Shachar Lovett for sharing some
of the unpublished parts of their work \cite{LMS} and allowing
us to present parts of their proof in Appendix~\ref{app}.
Thanks to Alex also for numerous stimulating discussions from
the early stages of this work, and to Jakob Nordstr\"om for
bringing some of the authors together on this work.

\bibliographystyle{alpha}

\newcommand{\etalchar}[1]{$^{#1}$}

\appendix

\newcommand{\rank}{{\mathrm{rank}}}
\newcommand{\Perm}{{\mathrm{perm}}}

\ifnum\stoc=1
\diffdimsection
\fi

\ifnum\stoc=1
\gowerssection
\fi

\ifnum\stoc=1
\xorsection
\fi

\section{Tightness of main theorem}

\label{app}
In this section we show that our main theorem
cannot be improved asymptotically. Specifically,
we show that there is a constant $\alpha > 1/2$ such that
for
infinitely many $d$, for sufficiently large $n$,
there exists a function $f = f_{d,n}:\F_2^n\to\F_2$
that passes the degree $d$ AKKLR test (i.e., the $(d+1)$-flat
test) with probability $\alpha$ (i.e., strictly greater than
half) while being almost uncorrelated with degree $d$ polynomials.

Our example comes directly from the works of \cite{LMS,GT07}.
In particular, the function $f_{d,n}$ is simply the degree
$d+1$ symmetric polynomial over $n$ variables (defined formally
below). When $d+1$ is a power of two, then \cite{LMS,GT07} (who
in turn attribute the ideas to \cite{AlonBeigel}) already
show that
this function is far from every degree $d$ polynomial.
To complete our theorem we only need to show that this function
passes the $(d+1)$-flat test with probability noticeably greater
than $1/2$. \cite{LMS,GT07} also analyzed this quantity, but
the published versions only show that this function passes the
$(d+1)$-flat test
with probability $1/2 + \epsilon(d)$ where $\epsilon(d) \to 0$
as $d \to \infty$. However, it turns out that an early (unpublished)
proof by the authors of \cite{LMS} can be used to
show that the acceptance probability
is $1/2 + \epsilon$ where $\epsilon$ is an absolute constant.
For completeness we include a complete proof here.

We start with the definition of the counterexample functions.
For positive integers $d, n$ with $d \leq n$, let
$S_{d,n}:\F_2^n \to \F_2$ be  given by
$$S_{d,n}(x_1,\ldots,x_n) = \sum_{I \subseteq [n], |I| = d}
\prod_{i \in I} x_i.$$

\begin{theorem}
\label{thm:counter}
Let $d+1 = 2^t$ for some integer $t \geq 2$.
Then, for every $\epsilon > 0$, there exists $n_0$ such
that for every $n \geq n_0$, the following hold:
\begin{enumerate}
\item $\delta_d(S_{d+1,n}) \geq 1/2 - \epsilon$.
\item $\Rej_{d,d+1}(S_{d+1,n}) \leq 1/2 - 2^{-7} + \epsilon$.
\end{enumerate}
\end{theorem}

Theorem \ref{thm:counter} follows immediately from the following
two lemmas.

\begin{lemma}[{\protect \cite[Theorem 11.3]{GT07}}]
\label{lem:counter-1}
Let $d+1 = 2^t$ for some integer $t \geq 0$.
Then, for every $\epsilon > 0$, for sufficiently
large $n$, we have $\delta_d(S_{d+1,n}) \geq 1/2 - \epsilon$.
\end{lemma}

\begin{lemma}
\label{lem:counter-2}
For every $d \geq 3$ and $\epsilon > 0$,
for sufficiently large $n$, we have
$\Rej_{d,d+1}(S_{d+1,n}) \leq 1/2 - 2^{-7} + \epsilon$.
\end{lemma}

We prove Lemma~\ref{lem:counter-2} in the rest of this section.
We stress again that this approach is from an unpublished version
of \cite{LMS}, and we include it for completeness.

We start with some notation. For $x, a_1, \ldots,a_{d+1} \in \F_2^n$,
let $I(x , a_1, \ldots,a_{d+1}) = 1$ if the $(d+1)$-flat test rejects
$S_{d+1,n}$
when picking the affine subspace $x + \mspan(a_1,\ldots,a_{d+1})$.
Note that $\Rej_{d,d+1}(S_{d+1,n}) = \mathbb E_{x,a_1,\ldots,a_{d+1}}
[I(x, a_1, \ldots,a_{d+1})]$, where the expectation is taken over
$x, a_1,\ldots,a_{d+1}$ picked uniformly and independently from $
\F_2^n$, conditioned
on $a_1, \ldots, a_{d+1}$ being linearly independent.

\begin{lemma}
For $x, a_1, \ldots,a_{d+1} \in \F_2^n$, let $M \in \F_2^{(d+1)
\times n}$
be the matrix whose $i^{\rm{th}}$ row is $a_i$. Then
$I(x, a_1, \ldots,a_{d+1}) = 1$ if and only if $M \cdot M^T$ is of
full rank.
\end{lemma}

Note that the acceptance of the $(d+1)$-flat is independent of $x$;
this is explained in the proof.

\begin{proof}
For each $I \subseteq [n]$, $|I| = d+1$, we define the polynomial
$f_I(x) = \prod_{i \in I} x_i$. Note that
$$S_{d+1, n} = \sum_{I \subseteq [n], |I| = d+1} f_I.$$

Now, the $(d+1)$-flat test accepts $S_{d+1, n}$ when picking the
affine subspace $x + \mspan(a_1,\ldots,a_{d+1})$ if and only if
\begin{equation}
\label{eq:rejformula}
  \sum_{J \subseteq [d+1]} S_{d+1, n} (x + \sum_{j \in J} a_j) = 0.
\end{equation}

Note that the acceptance of the $(d+1)$-flat
test is independent of $x$. This is because $S_{d+1,n}$ is a degree $d+1$
polynomial, so it's $(d+1)$st derivative (the output of the $(d+1)$-flat
test) is constant.  Hence, the test accepts if and only if
\begin{equation}
  \sum_{J \subseteq [d+1]} S_{d+1, n} (\sum_{j \in J} a_j) = 0,
\end{equation}
which can be rewritten as
$$ \sum_{I \subseteq [n], |I| = d+1} \sum_{J \subseteq [d+1]}
f_I (\sum_{j \in J} a_j) = 0.$$

For a fixed $I$, we focus our attention on the expression
$\sum_{J \subseteq [d+1]} f_I (\sum_{j \in J} a_j)$.
By definition, this equals
$\sum_{J \subseteq [d+1]} \prod_{i \in I} (\sum_{j \in J} a_
{j,i})$
which in turn equals
$\sum_{J \subseteq [d+1]} (-1)^{d+1-|J|} \prod_{i \in I} (\sum_{j
\in J} a_{j,i})$
(since $1 = -1$ in $\F_2$).
By Ryser's formula, this equals
$\Perm(M_I)$, where $\Perm$ is the permanent function, and
$M_I$ is the $(d+1) \times (d+1)$ submatrix
of $M$ formed by the columns of $I$.

Thus, the left hand side of Equation~\eqref{eq:rejformula} equals
$$\sum_{I \subseteq [n], |I| = d+1} \Perm(M_I).$$

Since we are working over $\F_2$, we have that $\Perm(M_I) = \det
(M_I) = \det(M_I)^2$.
Thus,
\begin{align*}
\sum_{I \subseteq [n], |I| = d+1} \Perm(M_I) &= \sum_{I \subseteq
[n], |I| = d+1} \det(M_I)^2\\
&= \det(M M^T). \mbox{ \quad\quad\quad by the Cauchy-Binet formula}
\end{align*}

We thus conclude that $I(x, a_1, \ldots, a_{d+1}) = 1$ if and only if
$MM^T$ is nonsingular.

\end{proof}

We thus turn our attention to the probability
that for a randomly chosen matrix $M$,
the matrix $M \cdot M^T$ is of full rank.
We first note the following fact on the distribution
of $M\cdot M^T$ when $M$ is chosen uniformly
from the space of full rank matrices.

\begin{lemma}
Let $A,B \in \F_2^{(d+1) \times (d+1)}$ be random variables
generated as follows: $A$ is a symmetric matrix
chosen uniformly at random, and $B = M \cdot M^T$ where
$M$ is a random $(d+1) \times n$ matrix chosen uniformly
from matrices of rank $d+1$. Then the total variation
distance between $A$ and $B$ is $O(2^{d-n})$.
\end{lemma}
\begin{proof}
Let $E$ be the event that the rows of $M$, along with the vector $
\mathbf 1$ (the vector which
is $1$ in each coordinate) are all linearly independent.
Note that the probability of $E$ is at least $1 - 2^{d+1-n}$.
We will now show that the distribution of $B | E$ is $\exp(-n)$-close
to the distribution
of $A$. This will complete the proof.

Let the rows of $M$ be $a_1, \ldots, a_{d+1}$. We pick them one at at
time.
Having picked $a_1, \ldots, a_i$, the new entries of $B$ that get
determined by
$a_{i+1}$ are the entries $B_{i+1, j}$ for all $j \leq i+1$
(these determine the entries $B_{j, i+1}$).
If $a_{i+1}$ is picked uniformly from $\F_2^n$, then by the linear
independence of
$a_1, \ldots, a_i, \mathbf 1$, we see that the bits
\begin{itemize}
\item $B_{i+1, 1} = \langle a_{i+1}, a_1 \rangle$,
\item $B_{i+1, 2} = \langle a_{i+1}, a_2 \rangle$,
\item $\ldots$,
\item $B_{i+1, i} = \langle a_{i+1}, a_i \rangle$,
\item $B_{i+1, i+1} = \langle a_{i+1}, a_{i+1} \rangle = \langle a_{i
+1}, \mathbf 1 \rangle$,
\end{itemize}
are all uniformly random and independent, as required. However, since
we have conditioned on $E$,
$a_{i+1}$ is not picked uniformly from $\F_2^n$, but picked uniformly
from $\F_2^n \setminus \mspan\{a_1, \ldots, a_{i}, \mathbf 1\}$.
Still, this distribution of $a_{i+1}$ is $2^{i+1-n}$-close to the
uniform distribution over
$\F_2^n$, and as a consequence, the distribution of the bits $B_{i+1,
1}, \ldots, B_{i+1, i+1}$ is
$O(2^{d-n})$-close to the distribution of uniform and independent
random bits.

To summarize, the entries of the matrix $B$ are exposed in $d+1$
rounds. The bits $B_{i,j}$ for $j < i$ are exposed
in round $i$, and their distribution, conditioned on the bits exposed
in all the previous rounds, is $O(2^{d-n})$-close to that of uniform
and independent random bits. This implies the desired claim on the
distribution of $B$.
\end{proof}

The final lemma
shows that the random symmetric matrix
$A \in \F_2^{(d+1) \times (d+1)}$ is full rank with probability
bounded away from 1/2 by some constant independent of $d$.  This
seems to be a well-analyzed problem and
\cite[Theorem 4.14]{BCJKLR06} (see also \cite{BM88} for related work)
already proves this fact; in particular, they
show that if $k \geq 3$, a random symmetric $k$-by-$k$ matrix
over $\F_2$ is full rank with probability at most $7/16$.
For completeness, we include a simple
proof that establishes a weaker bound on the
probability of non-singularity.

\begin{lemma}
For $k \geq 4$,
the probability that a random symmetric matrix $A \in \F_2^{k \times
k}$ has
full rank is at most $1/2 - 2^{-7}$.
\end{lemma}

\begin{proof}
Let $A_i$ denote the $i \times i$ submatrix containing
the first $i$ rows and columns of $A$. We consider the
probability that $A$ has full rank, conditioned upon
various choices of $A_{k-1}$. The first claim below shows
that the probability of this event is at most half
if the rank of $A_{k-1}$ is either $k-1$ or $k-2$; and
zero if the rank of $A_{k-1}$ is at most $k-3$. We then
argue in the next claim that the probability that
$A_{k-1}$ has rank at most $k-3$ is bounded below by a positive
constant independent of $k$. The lemma follows immediately.

\begin{claim}
Fix $B \in \F_2^{(k-1)\times (k-1)}$. The following
hold:
\begin{enumerate}
\item If $A_{k-1} = B$ and $\rank(B) \leq k-3$, then
$\rank(A) < k$.
\item If $\rank(B) = k-1$ then
$\Pr_{A} [ \rank(A) = k | A_{k-1} = B ] \leq 1/2$.
\item If $\rank(B) = k-2$ then
$\Pr_{A} [ \rank(A) = k | A_{k-2} = B ] \leq 1/2$.
\end{enumerate}
\end{claim}

\begin{proof}
Note that for every $i$, we have $\rank(A_i) \leq \rank(A_{i-1}) + 2$
since $A_i$ may be obtained from $A_{i-1}$ by first adding a column
and then a row, and each step may increase the rank by at most $1$.
Part (1) follows immediately.

For part (2), fix $a_{k,1},\ldots,a_{k,k-1}$ and consider a random
choice of $a_{k,k}$. Since $A_{k-1}$ has full rank, there is a
unique linear combination of the
$k-1$ rows of $A_{k-1}$ that generates the row $\langle a_{k,1},
\ldots,a_{k,k-1} \rangle$. $A$ has full rank only if
$a_{k,k}$ does not
equal the same linear combination of $a_{1,k},\ldots,a_{k-1,k}$, and
the probability of this event is at
most $1/2$.

Finally for part (3), assume for notational simplicity that
$A_{k-2}$ has full rank and the $(k-1)$th row of $A$ is linearly
dependent on the first $k-2$ rows. Now consider the addition
of a $k$th row to $A_{k-1}$ consisting of $a_{k,1},\ldots,a_{k,k-1}$.
Note that a necessary condition for $A$ to have full rank is
that the newly added row is linearly independent of the first
$k-2$ rows of $A_{k-1}$ (otherwise, the rank of the first $k-1$
columns of $A$ is only $k-2$). But again (as in Part (2)), there is a
unique linear combination of the rows of $A_{k-2}$ that generates
the row $\langle a_{k,1},\ldots,a_{k,k-2} \rangle$. The probability
that $a_{k,k-1}$ equals this linear combination applied to the
$(k-1)$-th column of $A_{k-1}$ is at least $1/2$.
\end{proof}

\begin{claim}
$\Pr_A [ \rank(A_{k-1}) \leq k-3 ] \geq 2^{-6}$.
\end{claim}

\begin{proof}
We start with the subclaim that for every $\ell$, we have
$\Pr [ \rank(A_{\ell+1}) = \rank(A_{\ell}) | A_{\ell} ] =
2^{(\rank(A_{\ell}) - \ell - 1)}$.
To see this, let $I \subseteq [\ell]$ be such that $A_{\ell}$
restricted to rows in $I$ has full rank (and
so $|I| = \rank(A_{\ell})$).  Then $A_{\ell}$
restricted to rows and columns of $I$ also has full rank.  (All the rows not in $I$ are in the span of the rows that are in $I$, and thus, by symmetry, all columns not in $I$ are in the span of the columns in $I$.)
Fix $a_{\ell+1,j}$ for $j \in I$
and note that there is unique linear combination of the rows of $I$ in $A_{\ell}$
such that they yield $a_{\ell+1,j}$ for $j \in I$. This linear
combination determines a unique setting for the remaining
$\ell + 1 - |I|$ entries of the $(\ell+1)^{\rm{th}}$ row of $A_{\ell
+ 1}$,
if the the rank of $A_{\ell+1}$ is to equal the rank of $A_\ell$.
The probability of this unique setting occurring equals $2^{|I| -
\ell - 1}$.
The subclaim follows.

Now the claim follows easily.
Let $m$ be the smallest integer $\leq k-2$ such that
$\rank(A_{m}) \geq k-4$. If such an $m$ does not exist,
then $\rank(A_{k-1}) \leq k-3$. Otherwise, $m$ exists, $m \geq k-4$ and
$\rank(A_m) \leq k-3$.
Using the subclaim above, we have for every $\ell \in \{m,\ldots,k-2\}$,
it is the case that
$\Pr[\rank(A_{\ell+1}) = \rank(A_{\ell})| A_{\ell} ] \geq 2^{(k-4)- \ell - 1}$.
Combining the claims for $\ell \in [m,k-2]$ (and recalling that $m \geq k-4$),
we get $\Pr[\rank(A_{k-1}) \leq k-3 ] \geq 2^{-6}$.
\end{proof}

Given the claims above, the lemma follows immediately.
\end{proof}

\end{document}